\documentclass[english]{article}
\usepackage{amssymb,amscd,amsthm,amsmath,babel,amsfonts,young,wrapfig}

\usepackage{graphicx}
\theoremstyle{plain}
\newtheorem{theorem}{Theorem}[section]
\newtheorem{lemma}[theorem]{Lemma}
\newtheorem{corollary}[theorem]{Corollary}
\newtheorem{proposition}[theorem]{Proposition}

\theoremstyle{definition}

\newtheorem{definition}[theorem]{Definition}
\newtheorem{remark}[theorem]{Remark}

\newtheorem{example}[theorem]{Example}
\newtheorem*{example*}{Example}
\newtheorem*{definition*}{Definition}

\newtheorem{remexas}[theorem]{Remarks and Examples}

\numberwithin{equation}{section}

\newcommand\opn[2]{%
  \newcommand{#1}{\operatorname{#2}}}

\newcommand\NN{{\mathbb N}}

\opn\charac{char}
\opn\projdim{proj\,dim}
\opn\depth{depth}
\opn\rank{rank}
\opn\rlex{rlex}
\opn\LEX{Lex}
\opn\fine{end}
\opn\lcm{lcm}
\opn\modnuo{mod}
\opn\supp{supp}
\opn\Ass{Ass}
\opn\Proj{Proj}
\opn\Spec{Spec}
\opn\Soc{Soc}
\opn\reg{reg}
\opn\Md{Md}
\opn\dBorel{dBorel}
\opn\Borel{Borel}
\opn\Shad{Shad}
\opn\Span{span}
\opn\Hs{Hilb}
\opn\ini{in}
\opn\Gin{Gin}
\opn\grade{grade}
\opn\length{length}
\opn\width{width}
\opn\greg{greg}
\opn\Ht{ht}
\opn\GL{GL}

\opn\Ker{Ker}
\opn\coker{coker}
 
\opn\im{Im}
\opn\Hom{Hom}
\opn\Tor{Tor}
\opn\Ext{Ext}
\opn\id{id}
\opn\Id{Id}
\opn\Homst{^*Hom}
\opn\Extst{^*Ext}
\opn\Gamst{^*\Gamma}
\opn\Hst{^*H}

\opn\KRS{KRS}
\opn\BKRS{BKRS}
\opn\ASL{ASL}

\opn\lk{lk}
\opn\st{st}

\newcommand\apfa{[a_1,\ldots,a_{2t}]}
\let\:=\colon
\let\phi=\varphi

\newcommand\thh{^{th}}
\newcommand\thf{^{st}}
\newcommand\pnt{{\raise0.5mm\hbox{\large\bf.}}}

\newcommand\ov{\overline}
\newcommand\Ial{I_\alpha(X)}
\def\cocoa{\mbox{\rm
 C\kern-.13em o\kern-.07 em C\kern-.13em o\kern-.15em A}}
\opn\po{pol}
\opn\pol{^{\bf p}} 

\def\a{\alpha}

\begin{document}
\noindent  
{\Large Gr\"obner bases of ideals cogenerated by Pfaffians}

\vspace{1cm}
\noindent
{\bf Emanuela De Negri}. {\small Universit\`a di Genova, Dipartimento di Matematica, Via Dodecaneso 35, 16146 Genova, Italy. {\it e-mail}:
denegri@dima.unige.it}\\  
\noindent
{\bf Enrico Sbarra}. {\small  Universit\`a di Pisa, Dipartimento di
         Matematica, Largo Pontecorvo 5, 56127 Pisa, Italy.
         {\it e-mail}:  sbarra@dm.unipi.it} 

\vspace{1cm}

\noindent
{\small {\bf Abstract}

\noindent
We characterise the class of one-cogenerated Pfaffian ideals whose natural generators form a Gr\"obner basis with respect
to any anti-diagonal  term-order. We describe their initial ideals as well as the associated simplicial complexes, which 
turn out to be shellable and thus Cohen-Macaulay. We also provide a formula for computing their multiplicity.
}

\vspace{.3cm}
\noindent
{\small {\bf Keywords:} Gr\"obner bases, pure simplicial complexes, Pfaffian ideals, KRS.}

\vspace{.7cm}
\noindent
\section*{Introduction}
Let $X=(X_{ij})$ be a skew-symmetric $n\times n$ matrix of indeterminates. 
By \cite{DeP}, the polynomial ring $R=K[X]:=K[X_{ij}\: 1\leq i< j\leq n]$, $K$ being a field, is an algebra with straightening law (ASL for short) 
on the poset $P(X)$ of all Pfaffians of $X$ with respect to the natural  partial order defined in \cite{DeP}. Given any subset of $P(X)$, the ideal of $R$ it generates is called a  Pfaffian ideal. The special case of Pfaffian ideals $I_{2r}(X)$ generated by the subset $P_{2r}(X)$ of $P(X)$ consisting of all Pfaffians of size $2r$ has been studied extensively (\cite{A}, \cite{KL}, \cite{Ma}). These ideals belong to a wider family of Pfaffian ideals called  one-cogenerated or simply  cogenerated.
A cogenerated Pfaffian ideal of $R$ is an ideal generated by all Pfaffians of $P(X)$ of any size which are not bigger than or equal to a fixed Pfaffian $\alpha$. We  denote it by $\Ial:=(\beta \in P(X) \: \beta \not\geq \alpha)$. Clearly,  if the size of $\alpha$ is $2t$, then all Pfaffians of size bigger than $2t$ are in $\Ial$.
The ring $R_\a (X)=K[X]/I_{\a}(X)$ inherits the ASL structure from $K[X]$, by means of which one is able to prove  that $R_\a (X)$ is a Cohen-Macaulay normal domain, and characterise Gorensteiness, as performed in  \cite{D}.
In \cite{D2} a formula for the $a$-invariant of $R_\a (X)$ is also given.

Our attention will focus on the properties of cogenerated Pfaffian ideals and their Gr\"obner bases (G-bases for short) w.r.t. anti-diagonal term orders, which are natural in this setting.
By \cite[Theorem 4.14]{Ku} and, independently, by \cite[Theorem 5.1]{HT}, the set $P_{2r}(X)$ is a G-basis for the ideal $I_{2r}(X)$. In a subsequent remark the authors ask whether their result can be extended to any cogenerated Pfaffian ideal. 
This question is very natural, and in the analogue cases of ideals of minors of a generic matrix and of a symmetric matrix the answer is affirmative, as proved respectively in \cite{HT} and \cite{C}. Quite surprisingly the answer is negative (see Example \ref{emas}) and that  settles the starting point of our investigation.
The aim of this paper is to characterise  cogenerated Pfaffians ideals whose natural generators are a G-basis w.r.t. any anti-diagonal term order in terms of their cogenerator. We call such ideals G-Pfaffian ideals. 
In Section 1 we set some notation, recall some basic notions of standard monomial theory (cf. \cite{BV}), among which that of  standard tableau, and describe the Knuth-Robinson-Schensted correspondence (KRS for short) introduced and studied in \cite{K}, since this is the main tool used to prove results of this kind.
KRS has been first used by Sturmfels \cite{St}  to compute G-bases of determinantal ideals (see also \cite{BC}, \cite{BC2}) and it has been  applied in \cite{HT} to the study of Pfaffian ideals of fixed size.
It turns out that the original KRS is not quite right for our purposes, therefore the first part of Section 2 is devoted to the analysis of a modification that can
be applied to $d$-tableaux in a smart way. In the remaining part of the section  we state our main result, cf. Theorem \ref{main}, by characterising the class of G-Pfaffian ideals. This is performed by proving the two implications separately in  Theorem \ref{laltra} and  Proposition \ref{una} by means of what we call BKRS. In Section 3, Proposition \ref{inidescr},  we describe the initial ideals of such ideals and in Corollary \ref{inimini} their  minimal set of generators. Since these ideals are squarefree, we also study their associated simplicial complexes. By describing  faces and facets of the associated simplicial complex, Proposition \ref{facce} and Theorem \ref{facets} resp.,  we are able to prove that these complexes are pure, cf. Corollary \ref{purissimo}, and simplicial balls, whereas they are not simplicial spheres (see Corollary \ref{balla}). Furthermore, in Proposition \ref{molte} we provide a formula for computing their multiplicity. Finally, in Proposition \ref{shella}, we prove shellability, which yields that the simplicial complexes associated with G-Pfaffian ideals are Cohen-Macaulay as well. The interested reader can find other recent developments in the study of Pfaffian ideals in \cite{RSh} (cf. Remark  \ref{boh} (iv)) and \cite{JW} (see the end of the last section).
\noindent

\section{Standard monomial theory for Pfaffians and KRS}\label{KRSP}
Let $X=(X_{ij})$ be a skew-symmetric $n\times n$ matrix of indeterminates and let $R=K[X]:=K[X_{ij}\: 1\leq i< j\leq n]$ the polynomial ring over the field $K$. 
The Pfaffian $\alpha=\alpha(A)$ of a skew-symmetric sub-matrix $A$ of $X$ with row and column indexes $a_1<\ldots<a_{2t}$ is denoted by $[a_1,\ldots,a_{2t}]$. 
We say that $\alpha$ is a {\it $2t$-Pfaffian} and that the {\it size} of $\alpha$ is $2t$.
Let now $P(X)$ be the set of all Pfaffians of $X$ and let us recall the definition of partial order on $X$ as introduced in \cite{DeP}.
Let $\alpha=[a_1,\ldots,a_{2t}], \beta=[b_1,\ldots,b_{2s}]\in P(X)$. Then
$$\alpha\leq \beta \hbox{\;\;\; if and only if \;\;\;} t\geq s \hbox{ and } a_i\leq b_i \hbox{ for } i=1,\ldots,2s.$$

\begin{definition}
Let $\alpha\in P(X)$. The ideal of $R$ cogenerated by $\alpha$ is the ideal $$\Ial:=(\beta \in P(X) \: \beta \not\geq \alpha).$$
\end{definition}
We observe that  the ideal of $R$ generated by $P_{2r}(X)$, the set of all Pfaffians of size $2r$, is nothing but $\Ial$, where $\alpha=[1,\ldots,2r-2]$.  We recall that a {\it standard monomial} of $R$ is a product $\alpha_1\cdot\ldots\cdot \alpha_h$ of Pfaffians with $\alpha_1\leq\ldots\leq \alpha_h$.
Since $R$ is an ASL  on $P(X)$, standard monomials form a basis of $R$ as a $K$-vector space and,  since $\Ial$ is an order ideal, the ring $R_\a (X)$ inherits the ASL structure by that of $R$:

\begin{proposition}
\label{k-bases}
The standard monomials  $\alpha_1\cdot\ldots\cdot \alpha_h$ with $\alpha_1\leq\ldots\leq \alpha_h$ and $\alpha_1\not\geq\alpha$ form a $K$-basis of $\Ial$.
\end{proposition} 

\noindent
A natural way to represent monomials, i.e. products of Pfaffians,  is by the use of tableaux. Given $\alpha_i=[a_{1i},a_{2i},\ldots,a_{t_{i}i}]$ for $i=1,\dots,h$, one identifies a monomial  $\alpha_1\cdot\ldots\cdot \alpha_h$ with the tableau $T=|\alpha_1|\alpha_2|\ldots|\alpha_h|$, whose $i$-th column is filled with the indexes of the Pfaffian $\alpha_i$. 
Clearly, such a tableau has two properties: the size of all of its columns is even, i.e. $T$ is a {\it $d$-tableau}, and each column is a strictly increasing sequence of integers. We recall also that
the {\it shape} of $T$ is the vector $(\lambda_1,\ldots,\lambda_t)$ where $\lambda_j$ is the number of entries in the $j\thh$ row of $T$; the {\it length} of $T$ is simply the number of entries of the first column and it is denoted by $\length(T)$. Finally, $T$ is said to be {\it standard} if the elements in every row form a weakly increasing sequence. For instance, the standard monomial $\alpha_1\alpha_2\alpha_3=[1,2,3,4][2,3,4,5][2,6]$ is encoded into the standard tableau $T=|\alpha_1|\alpha_2|\alpha_3|$, as shown in Figure \ref{tab}.

\begin{wrapfigure}{r}{0.5\textwidth}
\begin{center}
\includegraphics[width=0.15\textwidth]{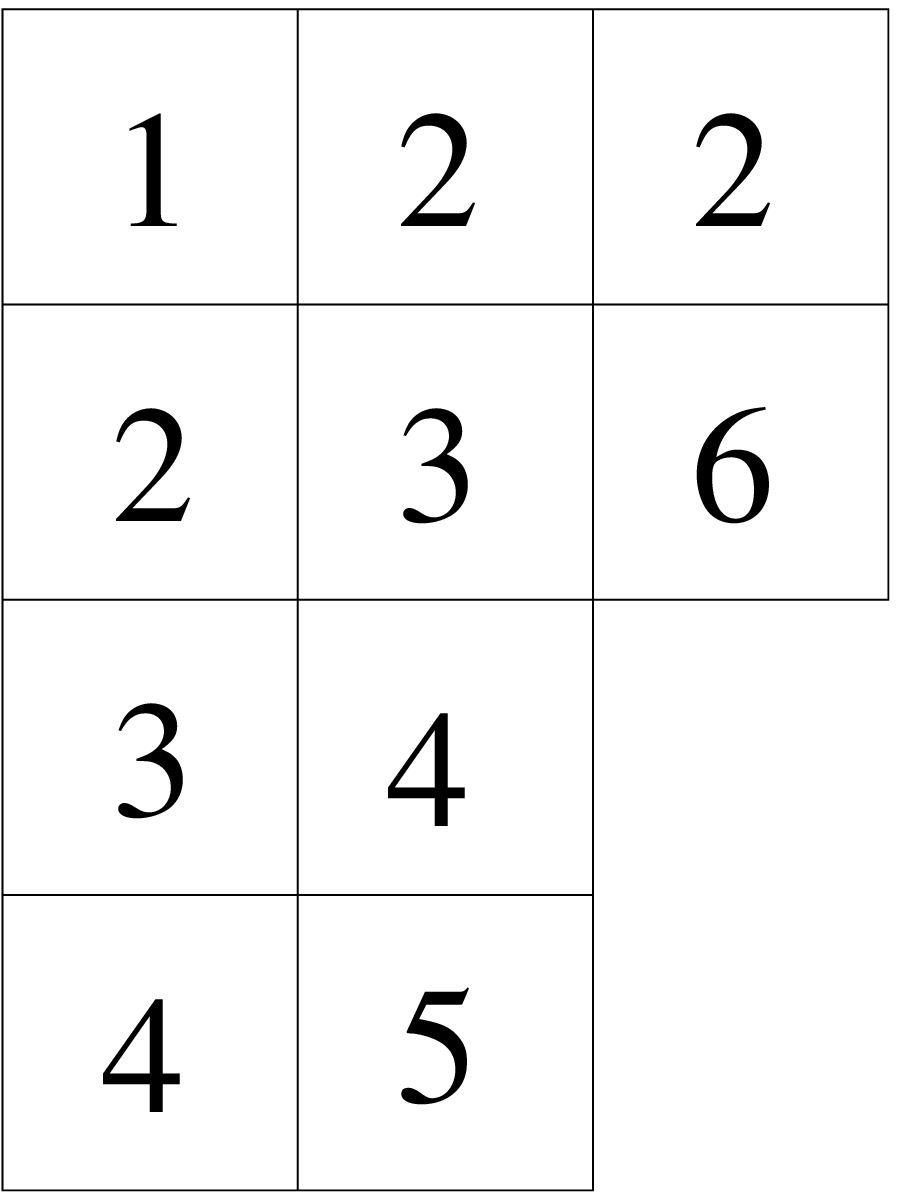}
\end{center}
\caption{$T=|\alpha_1|\alpha_2|\alpha_3|$.}
\label{tab}
\end{wrapfigure}
\noindent

\noindent
Thus, the tableau $T$  is a standard $d$-tableau with  shape $(3,3,2,2)$ and $\length(T)=4$. Obviously, if $T$ is a $d$-tableau of shape $(\lambda_1,\ldots, \lambda_t)$, $t$ is even and $\lambda_1=\lambda_2,\ldots,\lambda_{t-1}=\lambda_t$.
Also observe that a monomial  is standard if and only if the corresponding tableau is standard.

A very effective tool in studying G-bases of order ideals is the KRS correspondence.
For the reader's sake, we recall now the original KRS (cf. \cite{K}) as it is used in \cite{HT} to prove that the $2r$-Pfaffians are a G-bases of the ideal they generate. For more information on KRS the reader is also referred to \cite{F}. KRS is a bijection between the set of pairs of standard tableaux, which correspond naturally to standard monomials in the case of minors of a generic matrix, and ordinary monomials. Let  $(T_1,T_2)$ be an ordered pair  of standard tableaux of the same shape (a {\it standard bi-tableau} for short)  with $k$ elements each. One first associates with $(T_1,T_2)$ a two-lined  array $\begin{pmatrix} u_1 & u_2 & \ldots & u_k \\ v_1 & v_2 & \ldots & v_k \end{pmatrix}$ which satisfies the conditions $(\bullet)$:\;\;$u_1\geq\ldots\geq u_k$ and  $v_i\leq v_{i+1}$ if $u_i=u_{i+1}$. Such an array  can in turn  be identified with the monomial 
$f=\Pi_{i=1}^kX_{v_iu_i}$ of $R$.
The correspondence between tableaux and arrays relies on the {\bf delete} procedure we describe below.\\
{\bf delete}: It applies to a standard tableau $T$ and an element $u$, which is a corner of $T$, in the following way. Remove $u$  and set it in place of the first (strictly) smaller element of the above row going from right to left. Use the newly removed element in the same way, until an element $v$ is taken away from the first row of $T$. The result is a pair $(u,v)$ and a tableau $T'$ with exactly one element less than $T$.\\
Now, we describe the KRS.\\
{\bf $\KRS$}: Let $(T_1,T_2)$ be a bi-tableau. Take the largest element $u_1$ of $T_1$ with largest column index and remove it from $T_1$, obtaining a smaller tableau $T_1'$. Apply {\bf delete} to $T_2$ and the element of $T_2$ which is in the same position of $u_1$ in $T_1$, obtaining an element $v_1$ and a smaller tableau $T_2'$. Notice that this can be done because $u_1$ is placed in a corner of $T_1$ and that $T_1$ and $T_2$ have the same shape. The first column of the resulting  array is thus given by $u_1$ and $v_1$. Proceed in this way, starting
again with the bi-tableau $(T_1',T_2')$, until all of the elements are removed and the full sequence is achieved. By \cite{K}, the latter fulfils the desired conditions ($\bullet$).

\begin{example}\label{krsmadeeasy}
Let $T=|\alpha_1|\alpha_2|\alpha_3|$, with $\alpha_1=[1,3,4,5]$, $\alpha_2=[2,3]$ and $\alpha_3=[2,5]$. We apply the above procedure to the bi-tableau $(T,T)$:
\begin{figure}[here]
\begin{center}
\includegraphics[scale=0.39]{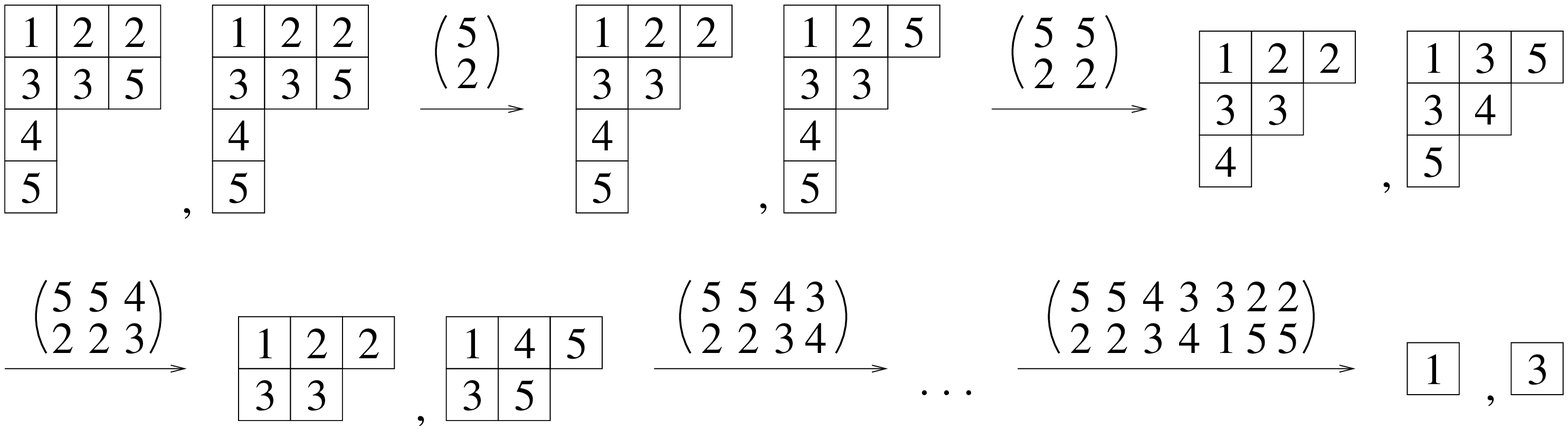}
\end{center}
\label{madeeasy}
\end{figure}

\noindent 
and, thus, $\KRS(T,T)=X_{25}X_{25}X_{34}X_{43}X_{13}X_{52}X_{52}X_{31}=X_{25}^4X_{34}^2X_{13}^2$. 
\end{example}

\noindent
In \cite{HT} such a correspondence is used in the  case of Pfaffians in the following manner. Given a standard monomial, one considers  its corresponding tableau $T$ and applies KRS to the standard bi-tableau of type $(T,T)$ obtaining the monomial $f=\KRS(T,T)$. We recall now the following definition.
\begin{definition}
The {\it width} of a monomial $f=\Pi_{i=1}^{k}X_{v_iu_i}$, with $u_1\ge u_2\ge\dots\ge u_k$,  is  the length of the longest increasing subsequence of $v_1,\ldots,v_k$ and it is denoted by $\width(f)$. 
\end{definition} 
\noindent
For instance, the longest increasing sequence in the previous example is $2,3,4,5$, therefore the width of the monomial is 4. By applying \cite[Theorem 3]{K} one has that $f=g^2$, where the essential data is contained in $g$, and the square appears because $T$ is used ``twice''. This is not an inconvenience in studying ideals generated by Pfaffians of a fixed size, since  the crucial point in the argument is the equality $\length(T)=\width(f)=2\width(g)$.\\
In our case the same holds (cf. Lemma \ref{width2}) but it is not sufficient to gather the information we need. Therefore we shall use a modified version of KRS that produces directly $g$ as an output and carries information on the 
indeterminates of $g$ as well. This is taken care of in the next section.



\section{A characterisation of G-Pfaffian ideals}\label{theorem}
Throughout this section and in the rest of the paper we shall consider  {\it anti-diagonal} term orders on $R$. We recall that a term order is said to be anti-diagonal  if
the initial monomial of the Pfaffian $[a_1,...,a_{2t}]$ is its main anti-diagonal ({\it adiag} for short), i.e.  
$$\ini([a_1,...,a_{2t}])=X_{a_1a_{2t}}X_{a_2a_{2t-1}}\cdot\ldots\cdot X_{a_ta_{t+1}}.$$ 
The aim of this section is to characterise what we call {\it G-Pfaffian} ideals, i.e. one-cogenerated ideals of Pfaffians whose natural generators are a G-bases w.r.t. such term orders. This is not always the case as it is shown in the following example.

\begin{example}\label{emas}
Let $X$ be a $6\times 6$ skew symmetric matrix of indeterminates,  and let $\Ial$ be the Pfaffian ideal cogenerated  by $\alpha=[1,2,4,5]$. The natural generators of $\Ial$ are 
$[1,2,3,4,5,6]$, $[1,2,3,4]$, $[1,2,3,5]$, $[1,2,3,6]$ whose leading terms are $[1,6][2,5][3,4]$, $[1,4][2,3]$, $[1,5][2,3]$, $[1,6][2,3]$ respectively. Therefore, the element $[1,2,3,4][1,5]-[1,2,3,5][1,4]$ belongs to $\Ial$ but its initial term, which is $[1,5][2,4][1,3]$, is not divisible by any of the leading terms of the generators. 
\end{example}

The main result of this  section is stated in  the following theorem.

\begin{theorem}\label{main}
The natural generators of $\Ial$ form a $G$-basis of $\Ial$ w.r.t. any anti-diagonal term order  if and only if $\alpha=\apfa$, with $a_i=a_{i-1}+1$ for $i=3,\ldots,2t-1$.
\end{theorem}

\noindent
For the purpose of proving the theorem, we use the following result about KRS, which is valid for any KRS correspondence. This is essentially due to Sturmfels \cite{St}. From now on we identify, with some abuse of notation, standard monomials with standard tableaux.

\begin{lemma}\label{Sturm}
Let $I\subset R$ be an ideal and let $B$ be a $K$-basis of $I$ consisting of standard tableaux. Let $S$ be a subset of $I$ such that for all $T\in B$ there exists $s\in S$ such that $\ini(s)|\KRS(T)$. Then $S$ is a G-basis of $I$ and $\ini(I)=\KRS(I)$.
\end{lemma}
\begin{proof}
See for instance that of \cite[Lemma 2.1]{BC}.
\end{proof}

\noindent
Now we describe the KRS correspondence we are going to use through the rest of the paper. This is a bijection  between standard $d$-tableaux and ordinary monomials, as introduced in \cite{Bu}.  This variant, which we denote  by  BKRS, makes a different use of the {\bf delete} procedure.\\
{\bf BKRS}: Consider the largest element of $T$ with largest column index, we say $u_1$, and its upper neighbour $u'$.  Remove $u_1$ from $T$ and  call the resulting tableau $T'$. Apply {\bf delete} to $T '$ and $u'$ to produce the element $v_1$. The output is $(u_1,v_1)$ and the tableau $T''$, and the first step is concluded. 
Evidently, one has  that $u_1>v_1$. Now we can start again with the tableau $T''$ and  proceeding in this fashion provides the sought after two-lined array.\\
Furthermore, as it has been shown in \cite[Section 2]{Bu}, the above array is ordered lexicographically and the correspondence is $1:1$. In this manner, one obtains a bijection between $d$-tableaux and two-lined arrays satisfying $(\bullet)$, which in turn can be identified with monomials of $R$.

\begin{example}\label{ehgia}
We compute $\BKRS(T)$, where $T$ is as in Example \ref{krsmadeeasy}:
\begin{figure}[here]
\begin{center}
\includegraphics[scale=0.39]{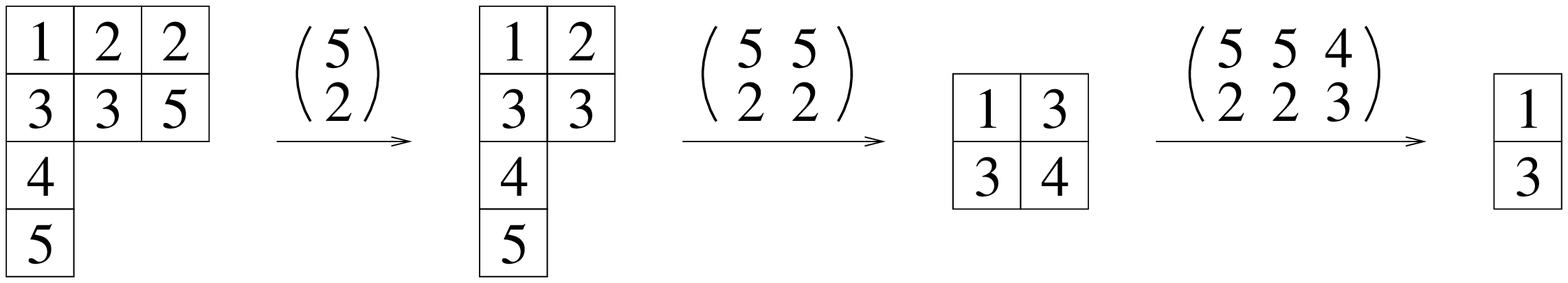}
\end{center}
\label{madeeasy2}
\end{figure}

\noindent 
and, thus, $\BKRS(T)=X_{25}^2X_{34}X_{13}$.
\end{example}
The fundamental connection between KRS and BKRS is yielded by the following proposition.

\begin{proposition}\label{KRS=B}
Let $T$ be a standard $d$-tableau. If $g=\BKRS(T)$ and $f=\KRS(T,T)$, then $f=g^2$.
\end{proposition}
\begin{proof}
See the proof in \cite[pag. 20]{Bu}.
\end{proof}

\noindent
As a consequence, one has the following result. 

\begin{lemma}\label{width2}
Let $f$ be the image of a standard $d$-tableau $T$ by the $\BKRS$ correspondence. Then $\width(f)=\length(T)/2.$
\end{lemma}
\begin{proof}
The proof follows by combining the results in \cite[Section 5]{HT} and Proposition \ref{KRS=B}.
\end{proof}
\noindent
For instance, in the last example it can be immediately seen that $\length(T)/2=2$, which is the length of the sequence $2,3$. The following remark about the BKRS procedure will be useful in the next proof.

\begin{remark}\label{1colo}
An element of the first column of a $d$-tableau $T$ is  moved when  all of the elements of the first column below it have been moved (and the bottom one deleted). In fact, an element of the first column of $T$ is moved only when replaced by its lower neighbour. To state this clearly, let $a$ be the element of $T$ in position $(i,1)$ and $b$ the element of the $(i+1)\thf$ row which is moved in position  $(i,1)$. Then $a<b$ and $b$ is smaller than any other element of the $i\thh$ row. Since $T$ is standard, $b$ must  belong to the first column.  
\end{remark}

Now that we have set our tools properly, our next task is to show that, for the class of cogenerated Pfaffian ideals described in Theorem \ref{main}, the natural generators form a G-basis.

\begin{theorem}\label{laltra}
Let $\alpha=\apfa$, with $a_i=a_{i-1}+1$ for $i=3,\ldots,2t-1$. The natural generators of $\Ial$ form a $G$-basis of $\Ial$.
\end{theorem}
\begin{proof}
By Proposition  \ref{k-bases} and Lemma \ref{Sturm} it is enough to prove that, given a $d$-tableau $T$ whose first column $\beta$ is in $\Ial$, there exists a Pfaffian in $\Ial$ whose initial term divides $\BKRS(T)$.\\
Let now $\beta=[b_1,\ldots,b_{2s}]\in \Ial$. Thus, $\beta\not\geq \alpha$ and one of the following must hold true: (i) $b_1<a_1$; (ii) $b_2<a_2$; (iii) $s=t$ and $b_{2t}<a_{2t}$; (iv) $s>t$.\\ We consider each case separately.\\
(i): $b_1$ is the smallest element of $T$, therefore by  BKRS it is paired with a bigger element, we say $c$. Thus,  $\ini(\Ial)\ni \ini([b_1,c])=[b_1,c]|\BKRS(T)$, as desired.\\
(ii): Without loss of generality we may exclude the trivial case when $T$ is just the one-columned tableau  $|\beta|$. 
Suppose there is an element $e$ which, when deleted, pushes the element $b_2$ into the first row in place of an element, we say  $c$, which is thus paired with $e$. By Remark \ref{1colo} we know that $e$ belonged to the
first column of $T$. As a consequence, if $d$ is paired with $b_2$, we have $c<b_2<d<e$. Therefore, since $b_2<a_2$, $[c,b_2,d,e]\in \Ial$ and its leading term divides $\BKRS(T)$.\\
(iii): The element $b_{2t}$ is the last element of the first column of $T$ and its row index is even. 
Let $T_k$ be the tableau occurring during the computation of $\BKRS(T)$ with the property that the biggest entry of $T_k$ with largest column index is $b_{2t}$. Let $T_{k+1}$ be the next tableau occurring in the procedure. Finally, let $f_k=\BKRS(T_k)$ and $f_{k+1}=\BKRS(T_{k+1})$. Evidently, $f_k=f_{k+1}X_{i_0b_{2t}}$ for some $i_0<b_{2t}$. Now, $\length(T_k)=2t$ and $\length(T_{k+1})=2t-2$, therefore by Lemma \ref{width2}, $\width(f_k)=t$ and $\width(f_{k+1})=t-1$. Thus, $f_{k+1}$ is divided by a monomial $X_{i_1j_1}\cdot\ldots\cdot X_{i_{t-1}j_{t-1}}$ ,  with  $i_1<\ldots<i_{t-1}<j_{t-1}<\ldots<j_1$ and, since the width of $f_k$ is one more than that of $f_{k+1}$,   $i_0<i_1$ and $j_1<b_{2t}$. Now, $[i_0,\ldots,i_{t-1},j_{t-1},\ldots,j_1,b_{2t}]$ is an element of $\Ial$, its initial term divides $f_k$ and $\BKRS(T)$, as desired.\\
(iv) Suppose now that $s>t$. If $f=\BKRS(T)$, then  $\width(f)=\length(T)/2=s$ by Lemma \ref{width2}. Thus, there exists a  $2s$-Pfaffian whose initial term divides $\BKRS(T)$, but all $2s$-Pfaffians are in $\Ial$. This concludes the proof of the last case and of the theorem.
\end{proof}

It is somehow surprising that the ideals satisfying the conditions of Theorem \ref{main}  are indeed the only ones endowed
with this property. We prove this fact next. In the proof,  we shall use the following standard expansion formula for Pfaffians: Given a $m\times m$ skew-symmetric submatrix $A=(a_{ij})$ of $X$, we denote by $A(i,j)$ the submatrix of $A$ obtained  by deleting the $i\thh$ and $j\thh$ row and column. Fixed an index $1\leq i \leq m$, we have
\begin{equation}\label{espanza}
\alpha(A)=\sum_{j=1}^m (-1)^{i+j+1}\sigma(i,j)a_{ij}\alpha(A(i,j))
\end{equation}
where $\sigma(i,j)$ is the sign of $j-i$.

\begin{proposition}\label{una}
Let $\alpha=\apfa\in P(X)$ and set $a_{2t+1}:=+\infty$, $i:=\min\{k\geq 2 \: a_k+1<a_{k+1}\}$. If $i<2t-1$, then the natural generators of $\Ial$ are not a G-basis for $\Ial$.
\end{proposition}

\begin{proof}
We prove that there exists an element which belongs to $\Ial$ and whose initial term is not divisible by the initial term of any Pfaffian in $\Ial$. 
In order to do so, we need to distinguish two cases.\\
Case 1: $i$ is even.\\
We start by observing that $i+2$ is even and $\leq 2t$ since $i<2t-1$.  We thus may 
let $\beta_1:=[a_1,a_2,\ldots,a_i,a_i+1,a_{i+1}]$, $\gamma_1:=[a_1,a_{i+2}]$, $\beta_2:=[a_1,a_2,\ldots,a_i,a_i+1,a_{i+2}]$, $\gamma_2:=[a_1,a_{i+1}]$ and consider
the element $\beta_1\gamma_1-\beta_2\gamma_2$.
Expanding $\beta_1$ and $\beta_2$ along  the $a_1\thf$ row by means of \eqref{espanza}, we obtain that $\beta_1$ is the alternating sum of $[a_1,a_{i+1}][a_2,\ldots,a_i,a_i+1]$, $[a_1,a_i+1][a_2,\ldots,a_i,a_{i+1}]$ and terms which do not contain either $[a_1,a_{i+1}]$ or $[a_1,a_i+1]$.  Similarly, $\beta_2$ is the alternating sum of $[a_1,a_{i+2}][a_2,\ldots,a_i,a_i+1]$,
$[a_1,a_i+1][a_2,\ldots,a_i,a_{i+2}]$ and terms which do not contain $[a_1,a_{i+2}]$. We observe that $[a_1,a_{i+2}]$ is the largest indeterminate which appears in $\beta_i, \gamma_i$, $i=1,2$, and $[a_1,a_{i+2}]>[a_1,a_{i+1}]>[a_1,a_i+1]>[a_1,a_i]>\ldots$. A quick verification on the sign of the summands shows that  a simplification occurs and it turns out that
\begin{eqnarray*}
\ini(\beta_1\gamma_1-\beta_2\gamma_2)&=&\ini([a_1,a_{i+2}][a_1,a_i+1][a_2,\ldots,a_i,a_{i+1}])\\
&=&[a_1,a_{i+2}][a_1,a_i+1][a_2,a_{i+1}]\ini[a_3,\ldots,a_i].
\end{eqnarray*}
This is an element of $\ini(\Ial)$, since $\beta_1,\beta_2\in \Ial$, and we identify it with the array
\begin{equation}\label{krs}
\left(\begin{array}{ccccccc} a_{i+2}& a_{i+1} & a_i+1 & a_i &\ldots &a_{\frac{i+4}{2}}\\
                             a_1 & a_2 & a_1 & a_3 & \ldots & a_{\frac{i+2}{2}}
\end{array}\right).
\end{equation} 
We now search for all Pfaffians $f$ such that $\ini(f)$ divides this monomial and show that they are not in $\Ial$.
Our task is reduced to merely considering all Pfaffians that one can build choosing sequences of growing indexes in the second row of \eqref{krs} and the reverse of the corresponding sequence which is determined in the first row of \eqref{krs} by this choice. It is immediate to see that no Pfaffian of size $i+4$ can be built this way. 
Moreover, the only such Pfaffian of size $i+2$ is $\ov{\alpha}:=[a_1,a_2,\ldots,a_i,a_{i+1},a_{i+2}]>\alpha$. As for those of size $i$, the only one which is not a sub-Pfaffian of $\ov{\alpha}$ is $\beta=[a_1,a_3,\ldots,a_i,a_i+1]>\alpha$. Since all the other Pfaffians   are sub-Pfaffians of $\ov{\alpha}$ or of $\beta$, and thus bigger than $\alpha$, the proof of this case is complete.\\
\noindent
Case 2: $i$ is odd.\\
Since $i<2t-1$, 
$i+3\leq 2t$. Thus, we may 
let $\beta_1:=[a_1,a_2,\ldots,a_i,a_i+1,a_{i+1},a_{i+3}]$, $\gamma_1:=[a_2,a_{i+2}]$, $\beta_2:=[a_1,a_2,\ldots,a_i,a_i+1,a_{i+2},a_{i+3}]$, $\gamma_2:=[a_2,a_{i+1}]$. Recalling that $a_i+1<a_{i+1}<a_{i+2}$, by computing as before we obtain
\begin{eqnarray*}
\ini(\beta_1\gamma_1-\beta_2\gamma_2) &=& [a_1,a_{i+3}][a_2,a_{i+2}][a_2,a_i+1]\ini[a_3,\ldots,a_i,a_{i+1}],
\end{eqnarray*}
its corresponding array being
\begin{equation*}
\left(\begin{array}{ccccccc} a_{i+3} & a_{i+2} & a_{i+1} & a_i+1 & a_i & \ldots & a_{\frac{i-3}{2}}\\ 
                                    a_1 & a_2 & a_3 & a_2 & a_4 & \ldots & a_{\frac{i+3}{2}} \end{array}\right),
\end{equation*}
and, by repeating the arguments of the previous case, we reach the conclusion in a  similar fashion. 
\end{proof}

\begin{remexas}\label{boh} 
{\bf (i)} Let us consider an interesting subclass of G-Pfaffian ideals. Let $\alpha=[1,2,\dots,2t-1,b]$. Then $\Ial$  is generated by the $2t$-Pfaffians indexed in the first $b-1$ rows and columns of $X$ and by all $(2t+2)$-Pfaffians. Note that, despite the fact that there are several expansion formulas for Pfaffians (cf. for instance \cite{S}, \cite{Ku}), some of which resembling well-known expansion formulas for minors,  it is not possible to expand all of the $2t$-Pfaffians by means of those in the first $b-1$ rows and columns only. Thus, even in this simple case, the minimal set of generators of $\Ial$ contains Pfaffians of different sizes.\\
{\bf (ii)} Let $\Ial$ and $I_{\beta}(X)$ be G-Pfaffian ideals. By Lemma  \ref{Sturm} it is easy to prove  (see also \cite{BC2}) that the generators of $\Ial$ and of $I_{\beta}(X)$ together form  a G-basis of $\Ial+I_{\beta}(X)$  w.r.t. any anti-diagonal term order (note that, in general, $\Ial+I_{\beta}(X)$ is not a cogenerated ideal).\\
{\bf (iii)} The G-Pfaffian ideals we considered in {\bf (i)} belong to the class of generalised ladder Pfaffian ideals, as introduced  in \cite{DGo}. By using  linkage Gorla, Migliore and Nagel  \cite{GoMiN} proved that the natural generators of such ideals  form a G-basis w.r.t. any anti-diagonal term order. Observe that the only generalised ladder ideals among G-Pfaffian ideals are those considered in {\bf(i)}.\\
{\bf (iv)} In \cite{RSh} a class of Pfaffian ideals, containing cogenerated ideals, is considered. By using an approach which involves Schubert varieties, the authors describe initial ideals of the ideals in this class. In Remark 1.9.1 they emphasise the fact that the generators are not a G-basis w.r.t. the orders that they consider. 
\end{remexas}


\section{Multiplicity and shellability}\label{simplicial}
We start this section by proving an easy but useful reduction.
\begin{proposition}\label{riduciedimostra}
Let $X$, $R_X$, $\alpha$, $\Ial$ be an anti-symmetric matrix of indeterminates $(X_{ij})$ of size $n$, the ring $K[X]$, a Pfaffian $[a_1,\ldots,a_{2t}]\in P(X)$ and the Pfaffian ideal of $R_X$ cogenerated by $\alpha$ respectively. 
Moreover, let $X'$, $R_{X'}$, $\beta$ and $I_\beta(X')$ an anti-symmetric matrix of indeterminates $(X'_{hk})$ of size $n-a_1+1$, the ring $K[X']$, the Pfaffian $[1, a_2-a_1+1,\ldots,a_{2t}-a_1+1]$ and the Pfaffian ideal of $R_{X'}$ cogenerated by $\beta$ respectively. Then 
$$R_X/I_{\alpha}(X)\simeq R_{X'}/I_\beta(X').$$
\end{proposition}

\begin{proof}
If $a_1=1$ there is nothing to prove. Thus, let us assume that $a_1>1$ and observe that $X_{ij}\in \Ial$, for all $1\leq i\leq a_1-1$ and these are the only
indeterminates contained therein. The reader can easily see that the desired isomorphism is yielded by modding out these indeterminates and by a change of coordinates
that preserves the poset structure. 
\end{proof}

\noindent
As a consequence, in order to study numerical invariants of $\Ial$, we may without loss of generality assume that $a_1=1$. 
\begin{remark} {\it Computational issue}. The previous proposition also makes the computation of bigger examples, which are extremely resource intensive, possible. 
\end{remark}

Next, we describe initial monomial ideals of G-Pfaffian ideals.\\
\noindent
Having in mind what initial monomials  are w.r.t. anti-diagonal term orders, we say that a monomial 
$X_{i_1j_1}\cdot\ldots\cdot X_{i_{t}j_{t}}$ is a $t$-adiag if $i_1<i_2<\ldots<i_t<j_t<j_{t-1}<\ldots<j_1$. 
Let $\alpha=[1,a,\ldots,a+2t-3,b]\in P_{2t}(X)$. All the $(2t+2)$-Pfaffians belong to $\Ial$. 
All other Pfaffians in $\Ial$ are of size  $\leq 2t$ and are of type $[c,d,*,\ldots,*]$
 with $c<d\leq a-1$ or $[e_1,\ldots,e_{2t}]$ with $e_{2t}\leq b-1$.
Now observe that, for any Pfaffian $[c,d,*,\ldots,*,e,f]$, one has  
$\ini([c,d,e,f]) | \ini([c,d,*,\ldots,*,e,f])$ and, since the generators form a G-basis, then  $\Ial$ is generated by 
$$\{[c,d] \: c<d\leq a-1\}\cup \{[c,d,e,f] \: c<d\leq a-1\}\cup \{[e_1,\ldots,e_{2t}] \: e_{2t}\leq b-1\}.$$  
In this way we have proven the following result.

\begin{proposition}\label{inidescr}
The ideal $\ini(\Ial)$ is generated by\\
$(i)$ all of the indeterminates in the first $a-1$ rows and columns (region $A$);\\
$(ii)$ all of the $2$-adiags in the first $a-1$ rows (region $A\cup B\cup C$);\\
$(iii)$ all of the $t$-adiags in the first $b-1$ rows and columns (region $A\cup B\cup D$);\\
$(iv)$ all of the $(t+1)$-adiags.\\ 

\begin{figure}[here]
\begin{center}
\includegraphics[scale=0.2]{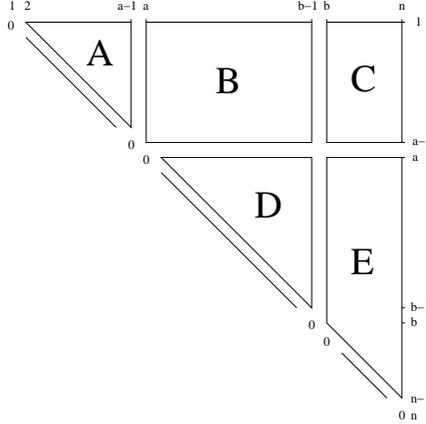}
\end{center}
\caption{$X_+$ and its regions.}
\label{regions}
\end{figure}
\end{proposition}

\begin{corollary}\label{inimini}
The minimal set of generators of $\ini(\Ial)$ is given by:\\
$(i)$ all of the indeterminates in $A$;\\ 
$(ii)$ all of the $2$-adiags in $B\cup C$;\\
$(iii)$ all of the $t$-adiags in $B\cup D$ with at most one indeterminate in $B$;\\
$(iv)$ all of the $(t+1)$-adiags outside of $A$, with at most one indeterminate in $B$ and at most $t-1$ indeterminates in $B\cup D$.
\end{corollary}

\noindent
Let $X_+:=\{(i,j) | 1\le i<j\le n\}$. In the rest of this section, with some abuse of notation, we shall identify $X_{ij}$ with $(i,j)$. Our next task is to describe the simplicial complex $\Delta_\alpha$ associated with $\Ial$, which is the family of all subsets $Z$ of $X_+$  such that the corresponding monomial $\Pi_{(i,j)\in Z}X_{ij}$  does not belong to $\ini(\Ial)$. The {\it modus operandi} is the same as in the classical cases: we decompose a set $Z$ into disjoint 
chains,  in a way similar to that of the ``light and shadow" procedure described in \cite{V} and used in \cite{HT}. Our proof is based on the use of two lights, say a ``sunlight" coming from the lower-left side and a ``moonlight" coming from the upper-right side, and a mixed use of such decompositions. Let us describe this in a more precise manner. Given a set $Z\subset X_+$, we let 
$$\begin{array}{c}\delta(Z):=\{(i,j)\in Z \: \nexists(i',j')\in Z \hbox{ with } i'>i, j'<j\},\\
\delta'(Z):=\{(i,j)\in Z \: \nexists(i',j')\in Z \hbox{ with } i'<i, j'>j\}.\end{array}$$ 
We also let  
$Z_1:=\delta(Z)$ and
$Z_1':=\delta'(Z)$. For $h>1$, we let 
$$Z_h:=\delta(Z\setminus \cup_{k<h}Z_k) \hbox{\;\;\;and\;\;\; } Z'_h:=\delta'(Z\setminus \cup_{k<h}Z'_k).$$ We thus obtain two decompositions of $Z=\cup_{h=1}^r Z_h$ and $Z=\cup_{k=1}^s Z'_k$, as disjoint chains, the first one corresponding to the sunlight shadows and the second one to the moonlight shadows. Note  that $r=s$, since the number of components in such a decomposition only depends on the maximal length of an anti-chain, i.e. an adiag, contained in $Z$; note also that if $P\in Z_i$ (resp. $Z'_j$) there is an adiag of length $i$ 
(resp. $j$) starting (resp. ending) in $P$. In the classical cases, the use of one light only is sufficient for describing the faces of the complex, but this does not work in our case. Therefore, we use a mixed decomposition: given a subset $Z$ of $X_+$ we decompose it as
$$Z=Z'_1\cup Z_1\cup Z_2\cup\ldots\cup Z_r,$$ where $Z'_1=\delta'(Z)$ and $Z_1\cup\ldots\cup Z_r$ is the sunlight decomposition of $Z\setminus Z'_1$. Furthermore,  if $Z_{t-1}\not=\emptyset$, we let $F\subset B$ be the subset of points in $X_+$ with column index larger than the smallest column index of a point in $Z_{t-1}$. We can now describe the faces of $\Delta_\alpha$.

\begin{proposition}\label{facce}
With the above notation, $Z\in \Delta_\alpha$ iff\\ 
- $Z\cap A=\emptyset$;\\
- $Z\setminus Z'_1\subseteq D\cup E$;\\
- $r\leq t-1$;\\
- if $r=t-1$ and $Z'_1\cap(B\cup D)\neq\emptyset$ then $Z'_1\cap F=\emptyset$.
\end{proposition}
\begin{proof}
$\Leftarrow$: We need to show that none of the forbidden adiags, i.e. those described in Proposition \ref{inidescr}, are in $Z$. 
If there were a $2$-adiag of $Z$ contained in $B\cup C$, there would exist a point $P\in (B\cup C)\setminus Z'_1$. Thus
$P\in (Z\setminus Z'_1)\setminus (D\cup E)$ and the second condition would be violated. If there were a $(t+1)$-adiag of $Z$ in $X_+\setminus A$ then $r\geq t$, since there would be a $t$-adiag in $Z\setminus Z'_1$. Suppose now that there exists a $t$-adiag $d$ of points of $Z$ in $B\cup D$: $d$ cannot be contained in $D$, since this would imply $r\geq t$. Thus $Z'_1\cap(B\cup D)\neq \emptyset$ and there exists a $(t-1)$-adiag in $D$, which implies $r=t-1$. We thus can assume that $Z'_1\cap F=\emptyset$. But, if we let $P_i$ be the points of $d$ such that $P_i\in Z_i$ with $i=1,\ldots,t-1$ and $P_t\in Z'_1$, we would also have 
$P_t\in Z'_1\cap F$, which is a contradiction. Now the conclusion is straightforward by a use of Proposition \ref{inidescr}.\\
$\Rightarrow$: Let $Z=Z'_1\cup Z_1\cup Z_2\cup\ldots\cup Z_r\in \Delta_\alpha$. Clearly, $Z\cap A=\emptyset$. If there existed $P\in (Z\setminus Z'_1)\setminus (D\cup E)$, we would have $P\in (B\cup C)\setminus Z'_1$. Thus, there would exist $Q\in Z'_1\cap (B\cup C)$ so that $Q,P$ formed a $2$-adiag in $B\cup C$, and this takes care of the second condition. Suppose now $r\geq t$. Then there exists
a $t$-adiag $P_1,\ldots,P_t\in D\cup E$ with $P_t\not\in Z'_1$. Therefore, we may prolong the $t$-adiag to a $(t+1)$-adiag in 
$X_+\setminus A$, which is not possible. Now we only have to show that, if $r=t-1$ and $Z'_1\cap(B\cup D)\neq \emptyset$ then $Z'_1\cap F=\emptyset$. If this were not the case, we could form a $t$-adiag in $B\cup D$ taking one point in each 
$Z_i$, $i=1,\ldots,t-1$ and one point in $Z'_1\cap F$. Again, this is not possible and we are done.
\end{proof}

Next, we describe the facets of $\Delta_\alpha$. We denote a saturated chain  of $X_+$ with starting point $Q$ and ending point $P$
simply by $QP$ and call it, as  is common in the literature, a {\it path}. In what follows $\sqcup$ denotes a  union of non-intersecting paths.

\begin{theorem}\label{facets}
Let $Q=(1,a)$, $Q_i=(a, a+2i-1)$ for $i=1,\ldots,t-2$ and $P_j=(n-2j+1,n)$ for
$j=1,\ldots,t$. Furthermore, let $Q_{t-1}=(a,k)$, $Q^h=(h,b)$, $P_{hk}=(h,k)$, with $h,k\in \NN$.  Then
$$Z \hbox{ is a facet of } \Delta_\alpha \hbox{ \phantom{aa} iff \phantom {aa}} Z=(QP_{hk} \sqcup Q^hP_t) \sqcup_{i=1}^{t-1}Q_iP_i,$$ for some $h\in\{1,\ldots,a-1\}$ and $k\in\{a+2t-3,\ldots,b-1\}$. 
\end{theorem}

\begin{figure}[here]
\begin{center}
\includegraphics[scale=0.2]{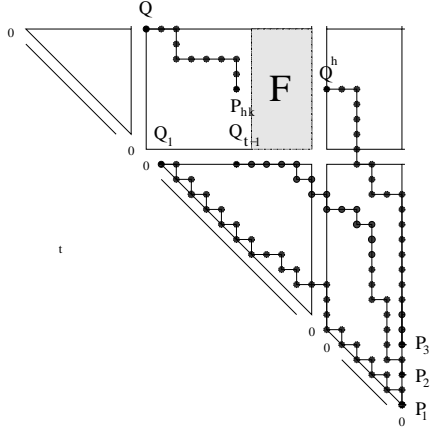}
\end{center}
\caption{A facet of $\Delta_\alpha$ with $t=3$.}
\end{figure}

\noindent
Before we proceed with the proof of the theorem, we need a couple of preparatory results. First, it might be useful to recall that, if $n=2m+i$ with $i\in\{0,1\}$, then the longest adiag contained in $X_+$ has length $m$. 
Therefore, being $b-a\geq 2t-2$, $D$ contains the $(t-1)$-adiag of the points $R_1=(a, b-1), R_2=(a+1,b-2),\ldots,R_{t-1}=(a+t-2,b-t+1)$. 

\begin{lemma}\label{lefacetssonociccie}
Let $Z=Z'_1\cup Z_1\cup\ldots \cup Z_r$ be a facet of $\Delta_\alpha$. Then, $r=t-1$.
\end{lemma}
\begin{proof}
For  $i=1,\ldots,t-1$ let $R_i$ be the points  described above. By contradiction, we let $r<t-1$  and we shall prove that $Z$ is not maximal. If $R_i\in Z$ for all
$i=1,\ldots,t-1$, $R_i$ all belong to different components of $Z$ and this implies that $R_1\in Z'_1$. By definition of $Z'_1$, we have $Z\cap C=\emptyset$. Now we let $R:=(a-1,b)\in C$ and the reader is left with the task to verify by means of Proposition \ref{facce} that $Z\cup \{R\}$ is a face of $\Delta_\alpha$. Else, if there exists $R_{i_0}\not\in Z$, again by Proposition \ref{facce}, $Z\subsetneq Z\cup\{R_{i_0}\}\in \Delta_\alpha$. In both cases the maximality of $Z$ is contradicted and the lemma is proven.  
\end{proof}

\begin{proof}[Proof of Theorem \ref{facets}]
It is immediately verified by means of Proposition \ref{facce} that the union $Z$ of such chains is a face of $\Delta_\alpha$. Also, it is easy to see that, given any point in $P\in X_+\setminus Z$, $Z\cup \{P\}$ is not a face of $\Delta_\alpha$, 
therefore $Z$ is a facet. 
Conversely, let $Z$ be a facet. By Lemma \ref{lefacetssonociccie}, we may decompose $Z=Z'_1\cup Z_1\cup\ldots Z_{t-1}$ and observe that, by maximality, each $Z_i$ must end exactly in $P_i$, for $i=1,\ldots,t-1$ and $Z'_1$ in $P_t$. Moreover, $Z'_1\subseteq B\cup C\cup E$: if there were a point $P\in Z'_1\cap D$ then, by Proposition \ref{facce}, $Z'_1\cap F=\emptyset$. Evidently, a point of $Z_{t-1}$ does not belong to $Z'_1$ thus a point of $Z_{t-1}$ with smallest column index is in the shadow of a point $P'\in Z'_1\cap C$, so that we would have a $2$-adiag $P',P$ contained in the same component $Z'_1$, which is a contradiction. Now, again by maximality, $Z'_1\cap B\neq\emptyset$, therefore, $Z'_1\cap F=\emptyset$ and $Z'_1$ can be seen as a path from $Q$ to a point $P_{hk}=(h,k)$ with $h\in\{1,\ldots,a-1\}$ and $k\leq b-1$ together with a path from a point $Q^{h'}=(h',b)$ to $P_t$. Since $Z$ cannot contain $2$-adiags in $B\cup C$, $h'\geq h$ and maximality forces $h'=h$ and $Q'=P_{hk}$. Finally, we consider $Z\setminus Z'_1$ and we follow the strategy in \cite[Theorem 5.4]{HT} to obtain that $Z_1,\ldots,Z_{t-2}$ are paths starting in $Q_i$, $i=1,\ldots,t-2$. Therefore $k\geq a+2t-3$ and $Z_{t-1}$ turns out to be a path from $(a,k)$, with $k\in\{a+2t-3,\ldots,b-1\}$. Now it is sufficient to collect all the data concerning the starting and ending points of the components to reach the  conclusion of this implication and of the proof.
\end{proof} 

\noindent
As a straightforward consequence, all of the facets of $\Delta_\alpha$ have the same cardinality, i.e. are of maximal
dimension.

\begin{corollary}\label{purissimo}
Let $\alpha$ be a G-Pfaffian and $\Delta_\alpha$ its associated simplicial complex. Then, $\Delta_\alpha$ is pure.
\end{corollary}
\begin{proof}
By Proposition \ref{riduciedimostra} we may assume that the first entry of $\alpha$ is $1$. By the previous theorem, we can compute the cardinality of the facets of $\Delta_\alpha$.
We have $k+h-a$ points in $QP_{hk}$, $2n-2t-h-b+2$ points in $Q^hP_t$, $2n-2t+4-a-k$ points in $Q_{t-1}P_{t-1}$ and
$\sum_{i=1}^{t-2}(2n-2a-4i+3)=(t-2)(2n-2a-2t+5)$ points in the other $t-2$ chains. Altogether these numbers add up 
to $d:=2nt-1-b-2(t-1)a-(2t-3)(t-1)$, which evidently does not depend on $h$ and $k$, therefore the complex is pure of dimension $d-1$.
\end{proof}

\noindent
We observe that, since $\dim K[\Gamma]=\dim \Gamma+1$ for any simplicial complex $\Gamma$,  the dimension computed in the above proof is the one predicted in \cite[(1.1)]{D}.

\begin{corollary}\label{balla}
$\Delta_\alpha$ is  a simplicial ball and not a simplicial sphere.
\end{corollary}
\begin{proof}
By means of Theorem \ref{facets}, one can easily verify that every sub-maximal face of $\Delta_\alpha$ is contained at most into two distinct facets of $\Delta_\alpha$. Moreover there is at least one sub-maximal face which is contained exactly in one facet. By \cite[4.7.22]{Bj et al.} these facts imply the assertion.
\end{proof}

Now we calculate the multiplicity of $R_{\alpha}(X)$ for a G-Pfaffian $\alpha$. For the sake of notational simplicity, we now modify slightly the notation introduced in Theorem \ref{facets} and we let $Q_t=(h,b)$. Furthermore, let  $A_{hk}=(a_{ij}^{hk})$ be the $t\times t$ matrix with entries $$a_{ij}^{hk}=\binom{x_{P_j}+y_{P_j}-x_{Q_i}-y_{Q_i}}{x_{P_j}-x_{Q_i}}-\binom{x_{P_j}+y_{P_j}-x_{Q_i}-y_{Q_i}}{x_{P_j}-y_{Q_i}},$$ where $x_p$ and $y_p$ denote the coordinates of a point $P$ in $X_+$. Observe that only the last two rows and columns of $A_{hk}$ are really dependent on $h,k$. By Proposition \ref{riduciedimostra}, it is
sufficient to provide a formula for a reduced G-Pfaffian.

\begin{proposition}\label{molte}
Let $\alpha=[1,a,\ldots,a+2t-3,b]$. Then 
$$e(R_{\alpha}(X))=\sum_{\substack{h=1,\ldots,a-1 \\ k=a+2t-3,\ldots, b-1}}\binom{h+k-a-1}{h-1}\,\,\det A_{hk}.$$
\end{proposition}
\begin{proof}
Since $\Delta_\alpha$ is pure by Corollary \ref{purissimo}, the multiplicity is just the cardinality of the set of all facets.
By Theorem \ref{facets},  these are $QP_{hk}\sqcup _{i=1}^{t}Q_iP_i$, with 
$h\in\{1,\ldots,a-1\}, k\in\{a+2t-3,\ldots,b-1\}$. Thus, the number we are seeking for is $\sum_{h,k} r(h,k)s(h,k)$ where $r(h,k)$ counts all saturated chains from $Q$ to $P_{hk}$ and $s(h,k)$ counts all disjoint unions of $t$ saturated chains $Q_iP_i$. It is well known that $r(h,k)=\binom{x_{P_{hk}}-x_Q+y_{P_{hk}}-y_Q}{x_{P_{hk}}-x_Q}$. Moreover, arguing as in \cite{GV} yields  that $s(h,k)=\det (b_{ij})$ where $b_{ij}$ is the number of saturated paths from $Q_i$ to $P_j$. Since all of these paths are contained in $X_+$, by \cite[Chap. 1, Theorem 1]{M}, $a_{ij}^{hk}=b_{ij}$ for all $i,j=1,\ldots,t$ and the proof is complete.
\end{proof}

\begin{example}
The multiplicity of the ring $R_\alpha(X)$, where $X$ is a $15\times 15$ skew-symmetric matrix and $\alpha=[4,8,9,12]$ is $50752$, as it can be easily computed using the previous proposition. For doing so, we consider the Pfaffian $[1,5,6,9]$, with $n=12$, $t=2$, $a=5$ and $b=9$, and compute the $12$ determinants of the $2\times 2$ matrices $A_{hk}$ where
$h=1,\ldots,4$ and $k=6,7,8$.
\end{example}

\noindent
We are now in a position of proving that $\Delta_\alpha$ is shellable and, thus, Cohen-Macaulay.

\begin{proposition}\label{shella}
$\Delta_\alpha$ is shellable.
\end{proposition}
\begin{proof}
Given $x=(u,v)\in X_+$, we set $\mathcal{R}_x:=\{(i,j)\in X_+\: i<u, j>v\}$, $\ov{\mathcal{R}}_x:=\{(i,j)\in X_+\: i\leq u, j\geq v\}$, $\mathcal{L}_x:=\{(i,j)\in X_+\: i>u, j<v\}$, and $\ov{\mathcal{L}}_x:=\{(i,j)\in X_+\: i\leq u, j\geq v\}$. Furthermore, if $S\subseteq  X_+$, we let $\mathcal{R}_S:=\cup_{x\in S}\mathcal{R}_x$, $\ov{\mathcal{R}}_S:=\cup_{x\in S}\ov{\mathcal{R}}_x$, $\mathcal{L}_S:=\cup_{x\in S}\mathcal{L}_x$ and $\ov{\mathcal{L}}_S:=\cup_{x\in S}\mathcal{L}_x$. Finally, we let $x^L:=(u+1,v-1)$. Our first task is to find a partial order on the set of the facets of $\Delta_\alpha$. Let $F=Z_1\cup\ldots\cup Z_{t-1}\cup Z_t$ and $F'=Z'_1\cup\ldots\cup Z'_{t-1}\cup Z'_t$ be decompositions of two facets $F$ and $F'$ of $\Delta_\alpha$, where we renamed the first 
components to $Z_t$ and $Z'_t$ for reason of notation. Now we may let
$$F \succeq F'  \hbox{\;\;\; if and only if \;\;\;} Z_i \subseteq \ov{\mathcal{R}}_{Z'_i} \hbox{\;\;\; for all \;\;\;} i+1,\ldots,t.$$
Next, we extend $\preceq$ to a total order on the set of facets of $\Delta_\alpha$. Thus, for the purpose of
 proving that $\Delta_\alpha$ is shellable, we need to show that, given any two facets $F\succ F'$, there exists
$x\in F\setminus F'$ and a third facet $F''$ such that $F \succ F''$ and $F\setminus F''=\{x\}$. 
Let $F \succ F'$ be two given facets decomposed as above. 
Since $F\not\preceq F'$,  we may consider the least 
integer $i$ such that $Z'_i\not\subseteq\ov{\mathcal{R}}_{Z_i}$, with $i\leq t-1$ or $i=t$.
 \begin{figure*}[here]
\begin{center}
\includegraphics[scale=0.2]{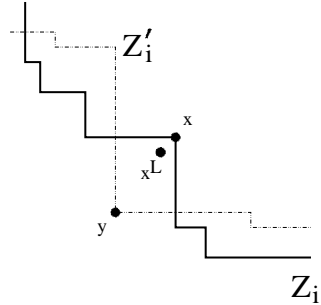}
\end{center}
\caption{Constructing a shelling on $\Delta_\alpha$.}
\label{shellaz}
\end{figure*}
\noindent In case $i\leq t-1$, if 
$y\in Z'_i\setminus \ov{\mathcal{R}}_{Z_i}$ then $y\in \mathcal{L}_{Z_i}$. Now we pick an element $x$ such that $y\in \mathcal{L}_x$ and $\ov{\mathcal{R}}_x\cap Z_i=\{x\}$  by taking an upper right corner of the chain $Z_i$ which is in $\mathcal{R}_y$, cf. Figure \ref{shellaz}. Finally let $F'':=(F\cup\{x^L\})\setminus\{x\}$. 
It is easy to verify that $x^L\in X_+\setminus F$ and that $F''$ is a facet of $\Delta_\alpha$ with the requested properties. In fact, $x^L\in F$ would imply $x^L\in Z_{i-1}\cap\ov{\mathcal{R}}_y$. Since $y\in Z'_i\subseteq \mathcal{R}_{Z'_{i-1}}$, $y$ is also in $\ov{\mathcal{R}}_{Z_{i-1}}$ because, by minimality of $i$, $Z'_{i-1}\subseteq \ov{\mathcal{R}}_{Z_{i-1}}$. Finally, this implies $x^L\in\mathcal{R}_{Z_{i-1}}$ and, consequently,
$x^L\not\in Z_{i-1}$, which is a contradiction.\\ In the other case, we let $Q^h=(h,b)\in Z_t$ and  start by arguing as before to find $y$. In constructing $x$, if we can choose a point of $F$ other than $Q^h$, we do it and the proof runs as in the previous case. Otherwise, we set $F'':=(F\cup \{P_{h+1k}\})\setminus \{Q^h\}$ 
and are left with the task to prove that $F''$ is a facet, which is equivalent to say that $P_{h+1k}\not\in F$. Since $P_{hk}$ is in $Z_t$, if $P_{h+1k}$
belonged to $F$, it would be a point of $Z_{t-1}$ and $h+1=a$. Therefore $y\in Z'_t\cap \mathcal{L}_{Q^h} \subseteq (B\cup C\cup E)\cap \mathcal{L}_{Q^h}=\emptyset$, which is the desired contradiction.
\end{proof}

In light of the results we obtained, it is now very natural to look for a new term order, w.r.t.  which the generators of any cogenerated ideal form a G-basis. An interesting non anti-diagonal term order is found in \cite{JW}, for which the result is proven for Pfaffians of fixed size and a complete description of the associated simplicial complexes is  given. It does not extend however to our setting (cf. again Example \ref{emas}).

\end{document}